\documentclass[12pt]{article}
\usepackage{amsfonts}
\usepackage{amsmath,amsthm}
\usepackage{cases}

\newtheorem{lemma}{Lemma}[section]

\newtheorem{theorem}{Theorem}[section]

\theoremstyle{definition}
\newtheorem{definition}{Definition}[section]
\newtheorem{corollary}[theorem]{Corollary}
\newtheorem{remark}{Remark}

\newcommand{\comment}[1]{}

\numberwithin{equation}{section}
\begin{document}
\title{Ultracontractivity and functional inequalities on infinite graphs}
\author{Yong Lin\footnotemark[1], Shuang Liu, Hongye Song}
\date{}
\maketitle

\renewcommand{\thefootnote}{\fnsymbol{footnote}}
\footnotetext[1]{Supported by the National Natural Science Foundation of China(GrantNo.$11271011$), and supported by the Fundamental Research Funds for the Central Universities and the Research Funds of Renmin University of China($11$XNI$004$).}

\begin{center}
\textbf{Abstract}
\end{center}
In this paper, we prove the equivalent of ultracontractive bound of heat semigroup or the uniform upper bound of the heat kernel with the Nash inequality, Log-Sobolev inequalities on graphs. We also show that under the assumption of volume growth and nonnegative curvature $CDE'(n,0)$
the Sobolev inequality, Nash inequality, Faber-Krahn inequality, Log-Sobolev inequalities, discrete and continuous-time uniform upper estimate of heat kernel are all true on graphs.

\section{Introduction}
One can consider the heat equation associated with the Laplace operator $\Delta$,
\begin{equation}\label{eq:heat}
\Delta u=\partial_{t}u
\end{equation}
which leads in general to a smoothing effect in the form of ultracontractivity. This means that, if $u(t,x)$ satisfies \eqref{eq:heat}, then there exist $\gamma(t)\rightarrow0$ as $t\rightarrow\infty$, such that for any $x\in V$, and $t>0$,
\[\|u(t,x)\|\leq\gamma(t)\|u(0,x)\|.\]
One may reformulate this by saying that the semigroup $P_t=e^{t\Delta}$ satisfies the estimate
\begin{equation}\label{eq:upper}
\|P_t\|_{1\rightarrow \infty}\leq\gamma(t).
\end{equation}
It turns out that there is a strong relationship between the geometry of $\Delta$ and the smoothing effect of the associated heat equation. The connection is made through functional inequalities, namely that they may be presented equivalently (up to constants) in various forms, such as families of log-Sobolev inequalities, Sobolev equalities, Nash equalities and Faber-Krahn inequalities. And it has been recently under extensive study, such as \cite{G} on manifolds, \cite{D89} on metric spaces and \cite{C} on graphs.

Given a measurable space $(E,\mathcal{F})$ with measure $\mu$. Let $\|f\|_p$ be the norm  of $f$ in $L^p(\mu)$, $1\leq p\leq \infty$. We say that $(E,\mathcal{F},\mu)$ satisfies the Sobolev inequality with constants $A\in \mathbb{R}, C>0$ if, for all integrable functions $f$,
\begin{equation}\label{eq:Sobolev}
\|f\|_p^2\leq A\|f\|_2^2+C\int_E|\nabla f|^2d\mu.
\end{equation}
where $\nabla f$ is the gradient of $f$. In Euclidean space $\mathbb{R}^n$, the exponent $p$ in \eqref{eq:Sobolev} will often take the form $p = \frac{2D}{D-2}$ for some $D>2$. The inequality entail remarkable smoothing properties of the semigroup in the form of ultracontractivity when $p = \frac{2D}{D-2}$. This result, due to Varopoulos \cite{V}, maybe established by different methods: Carlen, Kusuoka and Stroock used Nash inequalities \cite{CKS}, while Davies and Simon \cite{DS} use Log-Sobolev inequalities. In fact in the work of many authors such as Varopoulos, Grigor'yan, Bakry-Coulhon-Ledoux-Saloff-Coste, showed that Sobolev inequalities, Nash equalities and Faber-Krahn inequalities are all equivalent on a metric space and also on graphs (see \cite{C}).

For a $n$-dimensional Riemannian manifold $(M,g)$, Bakry and Emery~\cite{BE83} use the Bochner identity as a substitute for the lower Ricci curvature bound with $K$ on spaces where a direct generalization of Ricci curvature is not available, that is, for any $f\in C^\infty(M)$,
\begin{equation} \label{eq:cd-ine}
\frac{1}{2}\Delta|\nabla f|^2\geq\langle\nabla f,\nabla\Delta f\rangle+\frac{1}{n}(\Delta f)^2+K|\nabla f|^2.
\end{equation}
On graphs, both \cite{BHLLMY} and \cite{HLLY} introduce some modified curvature conditions. It is useful to estimate heat kernel on graphs by studying some properties under curvature condition.

In this paper, on the setting of graphs, we first give those different families of inequality, show the equivalence between them (Sobolev inequalities, Nash equalities, Faber-Krahn inequalities, and especially, Log-Sobolev inequalities), we prove these
by showing they are equivalent to the ultracontractity of heat semigroup or the uniform estimate of heat kernel. Then we introduce the curvature dimension condition on graphs, and study these above inequalities and properties under nonnegative curvature and volume growth assumption.

The paper is organized as follows: in section $2$, we give basic setting and main results of this paper. In section $3$, we prove the equivalent of log-Sobolev inequalities and the ultracontractity of heat semigroup. In section $4$, we prove the the equivalent of Nash equalities and the ultracontractity of heat semegroup, then the Theorem~\ref{th:equ}. In section $5$, we prove the Theorem~\ref{th:final}.

\section{Settings and main results}
Let us now introduce the necessary definitions and notations to state the results exactly. Let $G=(V,E)$ be a infinite graph. We allow the edges on the graph to be weighted, we consider a symmetric weight function $\omega: V\times V\rightarrow [0,\infty)$, the edge $xy$ from $x$ to $y$ has weight $\omega_{xy}>0$. In this paper, we assume this weight function is symmetric($\omega_{xy}=\omega_{yx}$). And the graph we are interested is locally finite, i.e. the degree of each vertex is finite as follows
$$m(x):=\sum_{y\sim x}\omega_{xy}<\infty, \quad \mbox{for any}~x\in V.$$
We define balls $B(x,r) = \{y\in V:d(x,y)\leq r\}$, and the volume of a subset $A$ of $V$, $V(A) =\sum_{x\in A}m(x)$. We will write $V(x,r)$ for $V(B(x,r))$.

We denote by $V^{\mathbb{R}}$ the space of real functions on $V,$ by $\ell^{p}(V)=\{f \in V^{\mathbb{R}}:\sum_{x\in V}m(x)|f(x)|^{p}<\infty\}, 1\leq p< \infty$,
the space of $\ell^{p}$ integrable functions on $V$ with respect to the degree $m$. For $p=\infty$, let $\ell^{\infty}=\{f \in V^{\mathbb{R}}:\sup_{x\in V}|f(x)|<\infty\}$ be the set of bounded functions. If $p=2$, let the inner product as $\langle f,g\rangle=\sum_{x\in V}m(x)f(x)g(x)$, then the space of $\ell^{2}$ is a Hilbert space. For all $1\leq p\leq \infty$, define $\ell^p$-norm by
\[\|f\|_p=\left(\sum_{x\in V}m(x)|f(x)|^p\right)^{\frac{1}{p}}, 1\leq p<\infty ~\mbox{and}~ \|f\|_{\infty}=\sup_{x\in V}|f(x)|.\]
And we denote by $C_c(V)\subset\ell^2$ the dense subset of functions $f\in V^{\mathbb{R}}$ with finite support.

For any function $f\in V^{\mathbb{R}}$ and any $x\in V$,, let Laplacian $\Delta:V^{\mathbb{R}}\rightarrow V^{\mathbb{R}}$ on $G$ be (the normalized graph Laplacian)
$$\Delta f(x)=\frac{1}{m(x)}\sum_{y\sim x} \omega_{xy}(f(y)-f(x)).$$
And the operator $\Delta$ is associated with the semigroup $P_{t}:V^{\mathbb{R}}\rightarrow V^{\mathbb{R}}$ by, for any function $f\in C_c(V)$,
$$P_{t}f(x)=\sum_{y\in V}m(y)p(t,x,y)f(y),$$
where $p(t,x,y)$ is so-called heat kernel with continuous time on infinite graphs (see \cite{KL12}, also \cite{W}), and $P_{t}f(x)$ is a solution of the heat equation. We know the operator $P_t$ is contractive, self-adjoint, and the semigroup property holds too in $C_c(V)$. We shall keep considering the discrete-time heat kernel $p_{k}(x,y)$ on $G$ because of its probabilistic significance, which is defined by
\[\left\{
  \begin{array}{ll}
    p_{0}(x,y)=\delta_{xy},  \\
    p_{k+1}(x,z)=\sum_{y\in V}p(x,y)p_{k}(y,z),
  \end{array}
\right.\]
where $p(x,y):=\frac{\omega_{xy}}{m(x)}$ is the transition probability of the random walk on the graph, and $\delta_{xy}=1$ only when $x=y$, otherwise equals to $0$.

For any positive function $f\in V^{R^+}$, we define the gradient form and the iterated gradient form by
\[2\Gamma(f)(x)=\frac{1}{\mu(x)}\sum_{y\sim x}\omega_{xy}(f(y)-f(x))^2,\]
\[2\Gamma_{2}(f) = \Delta\Gamma(f,g)-\Gamma(f,\Delta g)-\Gamma(\Delta f,g),\]
and the modified iterated gradient form by
\[\widetilde{\Gamma_2}(f)(x)= \Gamma_2(f)(x)- \Gamma\left(f, \frac{\Gamma(f)}{f}\right)(x).\]
Then define the curvature-dimension condition $CDE'(x,n,K)$ by
$$\widetilde{\Gamma_2}(f)(x) \geq \frac{1}{n} f(x)^2\left(\Delta \log f\right)(x)^2 + K \Gamma(f)(x),$$
we say that $CDE'(n,K)$ is satisfied on graphs if $CDE'(x,n,K)$ is satisfied for all $x \in V$.

In this paper, we say the graph satisfies a polynomial volume of growth $(V)$, that is for all $x\in V, r\geq 0$, with some $D>0$,
\begin{equation}\label{eq:v}
V(x,r)\geq cr^D.\tag{V}
\end{equation}
This condition is true in some Abelian Cayley graphs which satisfy the $CDE'(n,0)$.

We also need the following assumption $\Delta(\alpha)$ on graphs with loops on all vertices.
Let $\alpha>0$, $G$ satisfies $\Delta(\alpha)$ if, for any $x,y \in V$, and $x\sim y$,
\begin{equation*}\label{eq:a}
\omega_{xy}\geq \alpha m(x).
\end{equation*}
 The $\Delta(\alpha)$ was used in \cite{D99} and
other people before. This is a mild assumption. Since for a weighted graph $G$ without loop, we can add loops on every vertices and assign new weights on edges to get a graph $G_{\alpha}$ satisfying the $\Delta(\alpha)$, and the curvature assumption $CDE'$ is stable under this construction.

For simplification, we denote $\langle f\rangle=\sum_{x\in V}f(x)$, in this paper we will consider these inequalities on graphs.
\begin{definition}
Let $D>2$, we shall consider the following properties on $G$:
\begin{description}
  \item[(LS)] (Log-Sobolev inequality) $\langle f^2\log f\rangle-\|f\|_2^2\log\|f\|_2\leq \varepsilon\langle\Gamma(f)\rangle+\beta(\varepsilon)\|f\|_2^2$, where $\beta(\varepsilon)$ be a monotonically decreasing continuous function of $\varepsilon$, for all $\varepsilon>0$, for any $f$ function with finite support on $G$;
  \item[(S)] (Sobolev inequality)  $\|f\|_{\frac{2D}{D-2}}\leq c\langle \Gamma(f)\rangle$, for any $f$ function with finite support on $G$;
  \item[(N)] (Nash inequality) $\|f\|_2^{2+\frac{4}{D}}\leq c\langle \Gamma(f)\rangle\|f\|_1^{\frac{4}{D}}$, for any $f$ function with finite support on $G$;
  \item[(FK)] (Faber-Krahn inequality) $\lambda_1(\Omega)\geq cV(\Omega)^{-\frac{2}{D}}$, for every $\Omega$ finite subset of $G$, where $\lambda_1(\Omega)=\inf\left\{\frac{\langle\Gamma(f)\rangle}{\|f\|_2^2};supp(f)\subset\Omega\right\}$;
  \item[(FK)$^*$] (Relative Faber-Krahn inequality) $\lambda_1(\Omega)\geq \frac{c}{r^2}\left(\frac{V(x,r)}{V(\Omega)}\right)^{\nu}$, for all $x\in V, r\geq \frac{1}{2}, \nu>0$, $\phi\neq\Omega\subset B(x,r)$.
\end{description}
\end{definition}
And we also study a similar upper estimate of continuous-time heat kernel $p(t,x,y)$ and discrete-time heat kernel $p_k(x,y)$ separately. In fact, $p(t,x,y)$ is not an exact analogue of $p_k(x,y)$.
\begin{definition}
Two estimates of heat kernel on $G$ as follows :
\begin{description}
  \item[(CUE)] (Continuous-time uniform upper estimate) $ \sup_{x,y\in V}p(t,x,y)\leq Ct^{-\frac{D}{2}}$
  \item[(DUE)] (Discrete-time uniform upper estimate) $ \sup_{x,y\in V}\frac{p_k(x,y)}{m(x)}\leq Ck^{-\frac{D}{2}}$
\end{description}
\end{definition}

Studying heat kernel upper bounds and above inequalities are subject of great investigations for decades. Many authors (such as Varopoulos, Grigor'yan, Coulhon-Ledoux and so on) contributed to the development of this area. For example see \cite{C} on graphs.  The $(N)$ and $(S)$ are equivalent from H\"{o}lder inequality and the truncated functions technique (see \cite{BCLS}). Moreover, $(N)$ implies $(FK)$ by H\"{o}lder inequality, conversely, it's mainly due to Grigor'yan \cite{G}. The fact that $(N)$ is equivalent to $(DUE)$ in \cite{CKS}. In this paper we will prove that $(CUE)$ is equivalent to $(N)$ and $(LS)$ separately. We summarise the above results and our conclusions as follows.
\begin{theorem}\label{th:equ}
Let $D>2$, these properties are equivalent on graphs:
\begin{enumerate}
  \item Sobolev inequality $(S)$;
  \item Nash inequality $(N)$;
  \item Faber-Krahn inequality $(FK)$;
  \item Discrete-time uniform upper estimate $(DUE)$;
  \item Continuous-time uniform upper estimate $(CUE)$;
  \item Log-Sobolev inequality $(LS)$ with $\beta(\varepsilon)=c-\frac{D}{4}\log \varepsilon$.
\end{enumerate}
\end{theorem}
\begin{remark}
Note that the requirement $D>2$ is only necessary for the Sobolev inequality for ensure $\frac{2D}{D-2}>0$ in $(S)$, and not for the rest results. Actually, the proof below the ultracontractive bounds $(CUE)$ will transit through Nash inequalities $(N)$ (see Theorem \ref{th:nash}) and Log-Sobolev inequality $(LS)$ (see Theorem \ref{th:lsth}) and can be extended to any $D>0$.
\end{remark}

Another main purpose in this paper is to reveal a few of assumptions like nonnegative curvature and polynomial volume growth on graphs ensure the above properties. For any Abelian Cayley graph, the curvature-dimension condition $CDE'(n,0)$ (for example the lattice $Z^d$ with $CDE'(4.53d,0)$) and $(V)$ hold in the meantime for some appropriate constants with respect to $d$.
\begin{theorem}\label{th:final}
Let $D=D(n)>2$, assume a graph $G$ satisfies $CDE'(n,0)$, $\Delta(\alpha)$ and $(V)$, then all of these properties $(S), (N), (FK), (DUE), (CUE)$ hold with appropriate constants, and also $(LS)$ with $\beta(\varepsilon)=c(n)-\frac{D(n)}{4}\log \varepsilon$ hold.
\end{theorem}

\section{Log-Sobolev inequality and ultracontractivity on graphs}
In this section we consider the relationship between Log-Sobolev inequality and ultracontractive property. We say that the operator $P_t=e^{t\Delta}$ is ultracontractive if $P_t$ is bounded from $\ell^2$ to $\ell^\infty$ for all $t\geq0$. Let $\|A\|_{p\rightarrow q}$ be the norm of an operator $A$ from $\ell^p$ to $\ell^q$, that is $\|A\|_{p\rightarrow q}:=\sup_{f\in \ell^p}\frac{\|Af\|_q}{\|f\|_p}$. We have by duality for all $t>0$,
\[\|P_{\frac{t}{2}}\|_{2 \rightarrow \infty}=\|P_{\frac{t}{2}}\|_{1 \rightarrow 2}=\|P_{t}\|_{1 \rightarrow \infty}^{\frac{1}{2}},\]
indeed, this is because the semigroup property of the operator $P_t$ with $P_\frac{t}{2}\circ P_\frac{t}{2}=P_t$, the symmetric property $P_t^*=P_t$
as well as the following well-known equality
\[\|A^*A\|_{1\rightarrow\infty}=\|A\|_{1\rightarrow 2}^2.\]
Moreover
\[\|P_{t}\|_{1 \rightarrow \infty}=\sup_{x,y\in V}p(t,x,y),\]
that is to say, to get the ultracontractive property we mentioned before is same to estimate the upper bound of heat kernel $p(t,x,y)$.

Now we introduce the similar result on graph with Davies¡¯ theorem \cite{D89}.
\begin{theorem}\label{th:lsi}
For any $f\in\ell^2$, if the ultracontractivity
$$\|P_tf\|_\infty\leq e^{M(t)}\|f\|_2$$
satisfies with $M(t)$ is a continuous and decreasing function with $t$, then the logarithmic Sobolev inequality, for any $0\leq f\in C_c(V)$
\[\langle f^2\log f\rangle\leq \varepsilon\langle\Gamma(f)\rangle+\beta(\varepsilon)\|f\|_2^2+\|f\|_2^2\log\|f\|_2\]
holds with $\beta(\varepsilon)=M(\varepsilon)$ for any $\varepsilon>0$.
\end{theorem}
In fact there is a similar result in \cite{HLLY} (see Lemma 7.2), but it restrict the function field in $\ell^{\infty}(V,\mu)$. And the proof is basically the same. We simply reproduce them here for the sake of completeness.
\begin{proof}
For any $0\leq f\in C_c(V)$, since $\partial_t P_t f=\Delta P_t f$ for any $t\geq0$, and $p(s)$ is a bounded and continuous function with $s$ and its value more than or equal $1$. After simple computations, we have
\[
\partial_s\|P_sf\|_{p(s)}^{p(s)}=p'(s)\langle(P_sf)^{p(s)}\ln P_sf\rangle+p(s)\langle\Delta P_sf(P_sf)^{p(s)-1}\rangle
\]
If let $s=0$ in the above inequality, and let $p(s)=\frac{2t}{t-s}, 0\leq s<t$, then
\[\frac{d}{ds}\|P_sf\|_{p(s)}^{p(s)}\mid_{s=0}=\frac{2}{t}\langle f^2\ln f\rangle+2\langle f\Delta f\rangle.\]
We assume $\|f\|_2=1$, from the ultracontractivity and by the Stein interpolation theorem, we have
\[\|P_sf\|_{p(s)}\leq e^{\frac{M(t)s}{t}}.\]
From this point we can obtain
\[\frac{d}{ds}\|P_sf\|_{p(s)}^{p(s)}\mid_{s=0}\leq\frac{2M(t)}{t},\]
for observing $\|P_sf\|_{p(s)}^{p(s)}\mid_{s=0}=1$, $e^{\frac{M(t)sp(s)}{t}}\mid_{s=0}=1$, and
\[1\geq\lim_{s\rightarrow0^+}\frac{\|P_sf\|_{p(s)}^{p(s)}-1}{e^{\frac{M(t)sp(s)}{t}}-1}=\frac{d}{ds}\|P_sf\|_{p(s)}^{p(s)}\mid_{s=0}\frac{t}{2M(t)}.\]
Since the fact $-\langle f\Delta f\rangle=\langle\Gamma(f)\rangle$ from the symmetry of the weight of each edge, combining with the above equality, we obtain
\[\langle f^2\ln f\rangle\leq t\langle\Gamma(f)\rangle+M(t),~t>0.\]
If $\|f\|_2\neq1$, we put $f=\frac{g}{\|g\|_2}$ in the above inequality, and switch notation from $t$ to $\varepsilon$, yield the logarithmic Sobolev inequality we desire.
\end{proof}

Now we turn to the converse of the above result. First we introduce the following lemma. Similar result was proved by Varopoulos \cite{V} on smoothing setting.
\begin{lemma}\label{lem:LSI-p}
If there exists a monotonically decreasing continuous function $\beta(\varepsilon)$ such that for any $\varepsilon>0$ and $0\leq f\in C_c(V)$,
\begin{equation}\label{eq:LSI-2}
\langle f^2\log f\rangle\leq\varepsilon\langle\Gamma(f)\rangle+\beta(\varepsilon)\|f\|_2^2+\|f\|_2^2\log\|f\|_2.
\end{equation}
Then for all $2<p<\infty$,
\[\langle f^p\log f\rangle\leq\varepsilon\langle\Gamma(f^{p-1},f)\rangle+\frac{2\beta(\varepsilon)}{p}\|f\|_p^p+\|f\|_p^p\log\|f\|_p.\]
\end{lemma}
\begin{proof}
Putting $f=g^{\frac{p}{2}}$ ($2<p<\infty$) in \eqref{eq:LSI-2}, for all $0\leq g\in C_c(V)$, we obtain
\[\frac{p}{2}\langle g^2\log g\rangle\leq\varepsilon\langle\Gamma(g^{\frac{p}{2}})\rangle+\beta(\varepsilon)\|g\|_p^p+\frac{p}{2}\|g\|_p^p\log\|g\|_p,\]
we observe the following inequality between $\Gamma(g^{\frac{p}{2}})$ and $\Gamma(g^{p-1},g)$ with
$$\Gamma(g^{\frac{p}{2}})\leq\frac{p^2}{4(p-1)}\Gamma(g^{p-1},g),$$
by using Schwartz inequality,
\[(\alpha^{\frac{p}{2}}-\beta^{\frac{p}{2}})^2=\left(\int_\alpha^\beta\frac{p}{2}s^{\frac{p}{2}-1}ds\right)^2
\leq\frac{p^2}{4}(\alpha-\beta)\int_\alpha^\beta s^{p-2}ds
=\frac{p^2}{4(p-1)}(\alpha-\beta)(\alpha^{p-1}-\beta^{p-1}).\]
Then,
\[\frac{p}{2}\langle g^2\log g\rangle\leq\frac{\varepsilon p^2}{4(p-1)}\langle\Gamma(g^{p-1},g)\rangle+\beta(\varepsilon)\|g\|_p^p+\frac{p}{2}\|g\|_p^p\log\|g\|_p,\]
and switching the notation $g$ to $f$, then yield the result.
\end{proof}
The following theorem refers from Davies and Simon \cite{DS}.
\begin{theorem}\label{th:converse-LSI-p}
Let $\varepsilon(p)>0$ and $\delta(p)$ be two continuous functions defined for all $2<p<\infty$ such that
\[\langle f^p\log f\rangle\leq\varepsilon(p)\langle\Gamma(f^{p-1},f)\rangle+\delta(p)\|f\|_p^p+\|f\|_p^p\log\|f\|_p.\]
for any $0\leq f\in C_c(V)$. If
\[t=\int_2^\infty\frac{\varepsilon(p)}{p}dp,~~~~M=\int_2^\infty\frac{\delta(p)}{p}dp\]
are both finite, then
\[\|P_t\|_{2 \rightarrow \infty}\leq e^{M}.\]
\end{theorem}
\begin{proof}
Define the function $p(s)$ for $0\leq s<t$ by,
\begin{equation}\label{eq:p}
\frac{dp}{ds}=\frac{p}{\varepsilon(p)},~~~~p(0)=2,
\end{equation}
so that $p(s)$ is monotonically increasing and $p(s)\rightarrow\infty$ as $s\rightarrow t$. And another function $N(s)$ for $0\leq s<t$ satisfies
\[\frac{dN}{ds}=\frac{\delta(p)}{\varepsilon(p)},~~~~N(0)=0,\]
so that $N(s)\rightarrow M$ as $s\rightarrow t$.
We consider the functional $\log\left(e^{-N(s)}\|P_s f\|_{p(s)}\right)$, for any $0<s<t$ and any $0\leq f\in C_c(V)$. We obtain
\[\begin{split}
\frac{d}{ds}&\log\left(e^{-N(s)}\|P_s f\|_{p(s)}\right)
=\frac{d}{ds}\left(-N(s)+\frac{1}{p(s)}\log\|P_s f\|_{p(s)}^{p(s)}\right)\\
&=\frac{\delta(p)}{\varepsilon(p)}-\frac{1}{p^2}\frac{p}{\varepsilon(p)}\log\|P_s f\|_{p}^{p}+\frac{1}{p\|P_s f\|_{p}^{p}}\left(\frac{p}{\varepsilon(p)}\left\langle(P_sf)^{p}\log P_sf\right\rangle-p\langle\Gamma ((P_sf)^{p-1},P_sf)\rangle\right)\\
&=\frac{1}{\varepsilon(p)\|P_s f\|_{p}^{p}}\left(\left\langle(P_sf)^{p}\log P_sf\right\rangle-\varepsilon(p)\langle\Gamma ((P_sf)^{p-1},P_sf)\rangle-\delta(p)\|P_s f\|_{p}^{p}-\|P_s f\|_{p}^{p}\log \|P_s f\|_{p}\right)\\
&\leq 0.
\end{split}\]
So, for all $0\leq s<t$,
\[e^{-N(s)}\|P_s f\|_{p(s)}\leq \|f\|_2.\]
We can derive $\|P_t f\|_p^p$ be a decreasing function with respect to $t$, as follows
\[\partial_t\|P_t f\|_p^p=\langle p(P_t f)^{p-1}\Delta P_t f\rangle=-p\langle\Gamma((P_t f)^{p-1},P_t f)\rangle\leq-p\cdot\frac{4(p-1)}{p^2}\langle\Gamma((P_t f)^{\frac{p}{2}})\rangle\leq 0,\]
therefore combining the above two inequalities, we have for all $0\leq s<t$
\[\|P_t f\|_{p(s)}\leq \|P_s f\|_{p(s)}\leq e^{N(s)}\|f\|_2,\]
let $s\rightarrow t$, then
\[\|P_t f\|_\infty\leq e^{M}\|f\|_2.\]
If $0\leq f\in \ell^2$, there exists a sequence of $0\leq f_n\in C_c(V)$ such that $\|f_n-f\|_2\rightarrow0$ and let $f_n(x)\leq f(x)$ for any $x\in V$ (such as $f_n(x)=f(x)$ when $f(x)\leq n$, equal to $0$ otherwise for all $n\in \mathbb{N^+}$). Since $\|P_tf_n-P_tf\|_2\rightarrow0$ and from the above calculation, we have
\[\|P_t f_n\|_\infty\leq e^{M}\|f_n\|_2.\]
Therefore
\[\|P_t f\|_\infty\leq e^{M}\|f\|_2.\]
For a general $f\in \ell^2$ we know $|P_t f|\leq P_t|f|$ by the positivity of $P_t$, so
\[\|P_t f\|_\infty\leq \|P_t |f|\|_\infty\leq e^{M}\|f\|_2.\]
That completes what we desire.
\end{proof}

In the above Theorem, we can choose
\[\varepsilon(p)=\frac{2t}{p},~~~~\delta(p)=\frac{2\beta(\varepsilon(p))}{p},\]
then the solution of \eqref{eq:p} is
\[p(s)=\frac{2t}{t-s},\]
and
\[M=\int_2^\infty\frac{\delta(p)}{p}dp=\int_2^\infty\frac{2\beta(\varepsilon(p))}{p^2}dp=\frac{1}{t}\int_0^t\beta(\varepsilon)d\varepsilon=M(t).\]
Therefore, combining Lemma~\ref{lem:LSI-p} with Theorem~\ref{th:converse-LSI-p}, we can obtain the following result.
\begin{corollary}\label{co:ls}
Let $\beta(\varepsilon)$ be a monotonically decreasing continuous function of $\varepsilon$ such that for all $\varepsilon>0$ and $0\leq f \in C_c(V)$,
\[\langle f^2\log f\rangle\leq\varepsilon\langle\Gamma(f)\rangle+\beta(\varepsilon)\|f\|_2^2+\|f\|_2^2\log\|f\|_2.\]
If
\[M(t)=\frac{1}{t}\int_0^t\beta(\varepsilon)d\varepsilon\]
is finite for all $t>0$. Then $P_t$ is ultracontractive and for all $0<t<\infty$
\[\|P_t\|_{2 \rightarrow \infty}\leq e^{M(t)}.\]
\end{corollary}

Now, we give a example of the relationship between the bounds of $\|P_t\|_{2 \rightarrow \infty}$ and the efficiency of Log-Solobev inequality using Theorem \ref{th:lsi} and Corollary \ref{co:ls}.
If there exists constants $c_1>0$ and $N>0$ such that for all $t>0$,
\[e^{M(t)}\leq c_1 t^{-\frac{N}{4}},\]
then there exists a constant $c_2>0$ such that for all $\varepsilon>0$,
\[\beta(\varepsilon)\leq c_2-\frac{N}{4}\log \varepsilon.\]
Conversely, the above inequality implies that there exists a constant $c_3>0$ such that for all $t>0$,
\[e^{M(t)}\leq c_3 t^{-\frac{N}{4}}.\]
From the relationship between the upper bound of $p(t,x,y)$ and $\|P_t\|_{2 \rightarrow \infty}$ we mentioned before, we have the following conclusion.
\begin{theorem}\label{th:lsth}
For some constant $C>0$ and $N>0$ such that for all $t>0$,
\[\sup_{x,y\in V} p(t,x,y)\leq Ct^{-\frac{N}{2}}\]
is equivalent to the following Log-Solobev inequality, for some constant $C'>0$ and for all $\varepsilon>0$,
\[\langle f^2\log f\rangle\leq\varepsilon\langle\Gamma(f)\rangle+\left(C'-\frac{N}{4}\log \varepsilon\right)\|f\|_2^2+\|f\|_2^2\log\|f\|_2.\]
\end{theorem}

\section{Nash inequalities and ultracontractivity on graphs}
In this section we study the other important inequalities-Nash type inequalities and the above ultracontractive bounds estimate. This result goes back to Nash \cite{N} and further studied by Fabes and Stroock \cite{FS} on smoothing setting.
\begin{theorem}\label{th:nash}
Let $\mu>0$, the following two bounds are equivalent:
\begin{description}
  \item[(1)] Ultracontractivity property: for some constant $c_1>0$ and all $t>0$, $f\in\ell^2$,
\[\|P_t f\|_\infty\leq c_1 t^{-\frac{\mu}{4}}\|f\|_2.\]
  \item[(2)] Nash inequalities: for some constant $c_2>0$ and all $0\leq f\in C_c(V)$,
\[\|f\|_2^{2+\frac{4}{\mu}}\leq c_2\langle \Gamma(f)\rangle\|f\|_1^{\frac{4}{\mu}}.\]
\end{description}
\end{theorem}
\begin{proof}
First we introduce a similar equality from [BHLLMY] we will use later. For any $f\in C_c(V)$, and all $s>0$, from the facts that $P_t$ is self-adjoint, $P_t$ commutes with $\Delta$, and the semigroup property of $P_t$ (that is, $P_{\frac{t}{2}}\circ P_{\frac{t}{2}}=P_t$), we obtain
\[\begin{split}
\langle f,f\rangle
&-\langle P_s f,f\rangle=\langle f-P_s f,f\rangle=\sum_{x\in V}\mu(x)f(x)(P_0 f-P_s f)(x)\\
&=-\int_0^s\sum_{x\in V}\mu(x)f(x)\partial_tP_tf(x)dt=-\int_0^s\sum_{x\in V}\mu(x)f(x)\Delta P_t f(x)dt\\
&=-\int_0^s\sum_{x\in V}\mu(x)P_{\frac{t}{2}}f(x)\Delta P_{\frac{t}{2}} f(x)dt=\int_0^s\langle\Gamma(P_{\frac{t}{2}} f)\rangle dt.
\end{split}\]
Given (1) we have $\|P_t f\|_2\leq c_1 t^{-\frac{\mu}{4}}\|f\|_1$ by duality, then for all $f\in C_c(V)$,
\[\begin{split}
c_1^2 t^{-\frac{\mu}{2}}\|f\|_1^2
&\geq \|P_t f\|_2^2= \langle P_{2t} f,f\rangle\\
&=\langle f,f\rangle-\int_0^{2t}\langle\Gamma(P_{\frac{s}{2}} f)\rangle ds\\
&\geq\langle f,f\rangle-2t\langle\Gamma(f)\rangle,
\end{split}
\]
in the last step, we use that the function $\langle\Gamma(P_t f)\rangle$ is decreasing with respect to $t$, for any $t>0$, by
\[\begin{split}
\frac{d}{dt}\langle\Gamma(P_t f)\rangle
&=\frac{d}{dt}\left(\frac{1}{2}\sum_{x\in V}\sum_{y\sim x}\omega_{xy}(P_t f(y)-P_t f(x))^2\right)\\
&=\sum_{x\in V}\sum_{y\sim x}\omega_{xy}(P_t f(y)-P_t f(x))(\Delta P_t f(y)-\Delta P_t f(x))\\
&=2\langle\Gamma(P_t f,\Delta P_t f)\rangle=-2\langle\Delta P_t f,\Delta P_t f\rangle\leq0.
\end{split}\]
Therefore
\[\|f\|_2^2\leq2t\langle\Gamma(f)\rangle+c_1^2 t^{-\frac{\mu}{2}}\|f\|_1^2,\]
and (2) follows by putting
\[t=\langle\Gamma(f)\rangle^{-\frac{2}{\mu+2}}\|f\|_1^{\frac{4}{\mu+2}}.\]
Conversely given (2), for all $0\leq f\in C_c(V)$, we know the measure is invariant with $\|f\|_1=\|P_tf\|_1$ because of $\frac{d}{dt}\|P_tf\|_1=\sum_{x\in V}m(x)\Delta P_tf(x)=0$ from the definition of $\Delta$, so we have
\[-\frac{d}{dt}\|P_t f\|_2^2=\langle \Gamma(P_t f)\rangle\geq\frac{\|P_tf\|_2^{2+\frac{4}{\mu}}}{c_2\|P_tf\|_1^{\frac{4}{\mu}}}=\frac{\|P_tf\|_2^{2+\frac{4}{\mu}}}{c_2\|f\|_1^{\frac{4}{\mu}}}.\]
Therefore
\[-\frac{d}{dt}(\|P_t f\|_2^{-\frac{4}{\mu}})\geq\frac{2}{c_2\mu\|f\|_1^{\frac{4}{\mu}}},\]
and integrating the above inequality from $0$ to $t$, we obtain
\[\|P_t f\|_2^{-\frac{4}{\mu}}\geq\|P_t f\|_2^{-\frac{4}{\mu}}-\|f\|^{-\frac{2}{\mu}}\geq\frac{2t}{c_2\mu\|f\|_1^{\frac{4}{\mu}}}.\]
So
\[\|P_t f\|_2\leq\left(\frac{c_2\mu}{2t}\right)^{\frac{\mu}{4}}\|f\|_1=c_1t^{-\frac{\mu}{4}}\|f\|_1.\]
Finally, (1) follows by duality.

As before in proof of Theorem~\ref{th:converse-LSI-p}, in general $f\in \ell^2$, we have the same conclusion.
\end{proof}

\textit{The proof of Theorem \ref{th:equ}.}
From Theorem \ref{th:lsth} [$(CUE)\Leftrightarrow(LS)$ with $\beta(\varepsilon)=c-\frac{D}{4}\log \varepsilon$], Theorem~\ref{th:nash} [$(CUE)\Leftrightarrow(N)$], and $(S)\Leftrightarrow(N)\Leftrightarrow(FK)\Leftrightarrow(DUE)$ we illustrate before from summarising those authors' results,  we obtain what we desire.

\section{Nonnegative curvature}
In this section, we derive the above inequalities and upper estimate of the heat kernel on nonnegative curvature graphs.

From~\cite{CG}(see Theorem 5.4), for all locally finite graphs and two combining properties of the discrete-time on-diagonal upper estimate, that is
\begin{equation}\label{eq:onupper}
p_{k}(x,y)\leq \frac{cm(y)}{V(x,\sqrt{k})}
\end{equation}
in conjunction with the doubling property $DV(C)$, i.e.
\begin{equation*}
V(x,2r)\leq CV(x,r)
\end{equation*}
implied the relative Faber-Krahn inequality $(FK)^*$ with  $\nu=\frac{2}{\log_2 C}$.

Moreover, in the paper~\cite{HLLY} it proved \eqref{eq:onupper}(see Proposition 6.2) and $DV(C)$ with $C=C(n)$ (see Theorem 4.1) are both true if the graph satisfies  $CDE'(n,0)$ and $\Delta(\alpha)$. Therefore under the assumption with $CDE'(n,0)$ and $\Delta(\alpha)$, $(FK)^*$ holds with constant $\nu=\nu(n)=\frac{2}{\log_2 C(n)}$ ~( see\cite{CG}). Furthermore, note that if $(FK)^*$ in conjunction with the volume lower bound $(V)$ holds with $D=\frac{2}{\nu}=\log_2 C(n)$ , then one obtains the Faber-Krahn inequality $(FK)$~( see\cite{C}). For example, for any lattice $Z^d$ (satisfies $CDE'(4.53d,0)$), since $V(x,r)\simeq r^d$, then $(DV)$ holds for $C(n)=2^d$, and $(V)$ satisfied with $D=d$, which equal to the number $\log_2 C(n)$. So we have the following result.
\begin{theorem}\label{th:cde}
Assume a graph $G$ satisfies $CDE'(n,0)$, $\Delta(\alpha)$ and $(V)$ with $D=D(n)$, then the graph satisfies the Faber-Krahn inequality $(FK)$.
\end{theorem}

Combining Theorem \ref{th:equ} (the equivalence between functional inequalities and heat kernel estimate) with Theorem~\ref{th:cde} [$CDE'(n,0)+\Delta(\alpha)+(V)\Rightarrow(FK)$], so we prove the Theorem~\ref{th:final}.
\bibliographystyle{amsalpha}

Yong Lin,\\
Department of Mathematics, Renmin University of China, Beijing, China\\
\textsf{linyong01@ruc.edu.cn}\\
Shuang Liu,\\
Department of Mathematics, Renmin University of China, Beijing, China\\
\textsf{cherrybu@ruc.edu.cn}\\
Hongye Song,\\
Beijing International Studies University, Beijing, China\\
\textsf{songhongye@bisu.edu.cn}
\end{document}